\documentclass[a4paper,11pt]{article}
\usepackage{amsmath,amsthm,amssymb,enumitem,amscd,calc,titlesec}
\usepackage[bookmarks=false]{hyperref}
\usepackage{comment}

\renewcommand{\thesection}{\arabic{section}}
\titleformat{\section}{\Large\bf\boldmath}{\thesection.}{2ex}{}{}
\titlespacing{\section}{0ex}{2ex}{1ex}

\renewcommand{\thesubsection}{\arabic{section}.\arabic{subsection}}
\titleformat{\subsection}{\large\bf\boldmath}{\thesubsection.}{2ex}{}{}
\titlespacing{\section}{0ex}{1.5ex}{0.5ex}

{\theoremstyle{definition}\newtheorem{definition}{Definition}

\newtheorem{remark}[definition]{Remark}

}

\newtheorem{lemma}[definition]{Lemma}
\newtheorem{theorem}[definition]{Theorem}

\newtheorem{corollary}[definition]{Corollary}

\newcommand{\C}{\mathbb{C}}
\newcommand{\F}{\mathbb{F}}

\newcommand{\N}{\mathbb{N}}

\newcommand{\Z}{\mathbb{Z}}

\newcommand{\cE}{\mathcal{E}}
\newcommand{\recht}{\rightarrow}
\newcommand{\cU}{\mathcal{U}}

\newcommand{\eps}{\varepsilon}

\newcommand{\ovt}{\mathbin{\overline{\otimes}}}

\newcommand{\om}{\omega}

\newcommand{\cZ}{\mathcal{Z}}

\newcommand{\ot}{\otimes}

\newcommand{\cT}{\mathcal{T}}

\newcommand{\op}{^\text{op}}
\newcommand{\aff}{\operatorname{Aff}}

\begin{document}

\begin{center}
{\boldmath\LARGE\bf  Vanishing of the first continuous $L^2$-cohomology\vspace{0.5ex}\\ for II$_1$ factors}
\bigskip  %

{\sc by Sorin Popa\footnote{Mathematics Department, UCLA, CA 90095-1555 (United States), popa@math.ucla.edu\\
Supported in part by NSF Grant DMS-1101718} and Stefaan Vaes\footnote{KU~Leuven, Department of Mathematics, Leuven (Belgium), stefaan.vaes@wis.kuleuven.be \\
    Supported by Research Programme G.0639.11 of the Research Foundation~-- Flanders (FWO) and KU~Leuven BOF research grant OT/13/079.}}
\end{center}

\medskip

\begin{abstract}\noindent
We prove that the continuous version of the Connes-Shlyakhtenko first $L^2$-cohomology for II$_1$ factors, as proposed by A. Thom in \cite{Th06},
always vanishes.
\end{abstract}

In \cite{CS03},  A.\ Connes and D.\ Shlyakhtenko developed an $L^2$-cohomology theory for finite von Neumann algebras $M$, and more generally for weakly dense $^*$-subalgebras $A\subset M$ of such von Neumann algebras. Then in \cite{Th06}, A.\ Thom  provided an alternative, Hochschild-type characterization of the first such $L^2$-cohomology of $M$ as the quotient of the space of derivations $\delta:M\rightarrow \aff(M \ovt M\op)$ by the space of inner derivations, where $\aff(M \ovt M\op)$ denotes the $*$-algebra of operators affiliated with $M \ovt M\op$. Thom also proposed in \cite{Th06} a continuous version of the first $L^2$-cohomology, by considering the (smaller) space of derivations $\delta$  that are continuous from $M$ with the operator norm to $\aff(M \ovt M\op)$ with the topology of convergence in measure. He noted that in many cases (e.g., when $M$ has a Cartan subalgebra, or when $M$ is not prime), this cohomology vanishes, i.e.\ any continuous derivation of $M$ into $\aff(M \ovt M\op)$ is inner.

Following up on this work, V.\ Alekseev and D.\ Kyed have shown in \cite{AK11} that the first continuous $L^2$-cohomology also vanishes when $M$ has property (T), when $M$ is finitely generated with nontrivial fundamental group, or when $M$ has property Gamma. Recently, V.\ Alekseev proved in \cite{Al13} that this is also the case for the free group factors $L(\F_n)$.

In this article, we prove that in fact the first continuous $L^2$-cohomology vanishes for all finite von Neumann algebras. The starting point of our proof
is a key calculation in the proof of  \cite[Proposition 3.1]{Al13}, which provides a concrete sequence of elements $y_n$ in the II$_1$ factor
$M = L(\F_3)$ of the free group $\F_3$ with generators $a, b, c$, that tends to $0$ in operator norm, but
has the property that if a derivation $\delta : M \recht \aff(M \ovt M\op)$ satisfies $\delta(u_a) = u_a \ot \
1$ and $\delta(u_b) = \delta(u_c) = 0$, then $\delta(y_n)$ does not tend to $0$ in measure. More precisely, the $y_n$'s in \cite{Al13} are scalar multiples of  words
$w_n$ in $a, b, c$ with the property that $\delta(w_n)$ is a larger and larger sum of free independent Haar unitaries.
In the case of an arbitrary II$_1$ factor $M$, we fix a hyperfinite II$_1$ factor $R\subset M$
with trivial relative commutant, and then use
\cite{Po92} to ``simulate'' (in distribution) $L(\F_3)$ inside $M$, with $a$ any fixed unitary in $M$ and $b_m, c_m$ Haar unitaries  in $R$ such that $a, b_m, c_m$
are asymptotically free.
If now $\delta$ is a continuous derivation on $M$, then by subtracting an inner derivation, we may assume $\delta$ vanishes on $R$, thus on $b_m, c_m$.
If $\delta(a) \neq 0$, and if we formally define  $y_n$'s via the same formula as Alekseev's, with $a, b_m, c_m$ in lieu of $a, b, c$, then a careful estimation of norms of
$y_n$ and $\delta(y_n)$, which uses results in \cite{HL99}, shows that one still has $\|y_n\| \rightarrow 0$, while $\delta(y_n) \not\rightarrow 0$ in measure.

\vspace{4mm}

Let $M$ be a finite von Neumann algebra. We denote by $M\op$ the opposite von Neumann algebra and by $\aff(M \ovt M\op)$ the $*$-algebra of operators affiliated with $M \ovt M\op$. A \emph{derivation} $\delta : M \recht \aff(M \ovt M\op)$ is a linear map satisfying
$$\delta(ab) = (a \ot 1) \delta(b) + (1 \ot b\op)\delta(a) \quad\text{for all}\;\; a,b \in M \; .$$
For every $\xi \in \aff(M \ovt M\op)$, denote by $\partial \xi$ the \emph{inner} derivation defined as
$$(\partial \xi)(a) = (a \ot 1 - 1 \ot a\op)\xi \quad\text{for all}\;\; a \in M \; .$$

We equip $\aff(M \ovt M\op)$ with the \emph{measure topology,} i.e.\ the unique vector space topology with basic neighborhoods of $0$ given by
$$B(\tau_1,\eps) = \bigl\{\xi \in \aff(M \ovt M\op) \bigm| \exists \;\text{projection}\;\; p \in M\ovt M\op \;\text{with}\; \tau_1(p) > 1-\eps , \|\xi p\| < \eps \bigr\}$$
for all normal tracial states $\tau_1 : M \recht \C$ and all $\eps > 0$. If $\tau : M \recht \C$ is a normal \emph{faithful} tracial state, then $\{B(\tau,\eps) \mid \eps > 0\}$ is a family of basic neighborhoods of $0$ and there is no need to vary the trace.

\begin{theorem}\label{thm.inner}
Let $M$ be a finite von Neumann algebra.
Every derivation $\delta : M \recht \aff(M \ovt M\op)$ that is continuous from the norm topology on $M$ to the measure topology on $\aff(M \ovt M\op)$, is inner.
\end{theorem}

The following lemma is quite standard, but we include a detailed proof for completeness.

\begin{lemma}\label{lem.reduction}
It suffices to prove Theorem \ref{thm.inner} for II$_1$ factors $M$ with separable predual.
\end{lemma}
\begin{proof}
We prove the lemma in different steps.

{\it Step 1. It suffices to prove Theorem \ref{thm.inner} for diffuse, countably decomposable $M$.} Take a set $I = I_1 \sqcup I_2$ and an orthogonal family of projections $p_i \in \cZ(M)$ with $\sum_{i \in I} p_i = 1$ and such that for all $i \in I_1$, we have that $M p_i$ is countably decomposable and diffuse, and such that for all $i \in I_2$, we have that $M p_i$ is a matrix algebra. Since the projections $p_i$ are orthogonal, the element
$$\xi = \sum_{i \in I} (p_i \ot 1)\delta(p_i)$$
is well defined in $\aff(M\ovt M\op)$. The following direct computation shows that
\begin{equation}\label{eq.1}
\delta(p_k) = (\partial \xi)(p_k) \quad\text{for all}\;\; k \in I \; .
\end{equation}
Indeed,
\begin{equation}\label{eq.2}
(\partial \xi)(p_k) = (p_k \ot 1)\delta(p_k) - \sum_{i \in I} (p_i \ot p_k\op)\delta(p_i) \; .
\end{equation}
But, for all $i$ and $k$, we have
$$(1 \ot p_k\op)\delta(p_i) = \delta(p_i p_k) - (p_i \ot 1)\delta(p_k) \; .$$
Multiplying with $p_i \ot 1$ and summing over $i$, we find that
$$\sum_{i \in I} (p_i \ot p_k\op)\delta(p_i) = (p_k \ot 1)\delta(p_k) - \delta(p_k) \; .$$
In combination with \eqref{eq.2}, we find that \eqref{eq.1} holds.

Replacing $\delta$ by $\delta - \partial \xi$, we may assume that $\delta(p_i) = 0$ for all $i \in I$. It follows that $\delta(M p_i) \subset \aff(Mp_i \ovt (Mp_i)\op)$ for every $i \in I$. We denote by $\delta_i$ the restriction of $\delta$ to $M p_i$. By the assumption of step~1, $\delta_i$ is inner when $i \in I_1$. So for $i \in I_1$, we have that $\delta_i = \partial \xi_i$ for some $\xi_i \in \aff(M p_i \ovt (M p_i)\op)$. When $i \in I_2$, we have that $M p_i$ is a matrix algebra and we can take a complete system of matrix units $e^i_{jk}$ for $M p_i$. We then get that $\delta_i = \partial \xi_i$ for
$$\xi_i = \sum_k (e^i_{k1} \ot 1) \delta_i(e^i_{1k}) \; .$$
The vector $\xi = \sum_{i \in I} \xi_i$ is well defined in $\aff(M \ovt M\op)$ and $\delta = \partial \xi$.

{\it Step 2. It suffices to prove Theorem \ref{thm.inner} when $M$ is diffuse and has separable predual.} Using step~1, we may already assume that $M$ is diffuse and admits a faithful normal tracial state $\tau$ that we keep fixed. We start by proving the following three statements, using that $M$ admits the faithful trace $\tau$.

{\it If $\xi \in \aff(M \ovt M\op)$, there exists a von Neumann subalgebra $N \subset M$ with separable predual such that $\xi \in \aff(N \ovt N\op)$.} Take an increasing sequence of projections $p_n \in M \ovt M\op$ such that $\tau(1-p_n) \recht 0$ and $\xi p_n \in M \ovt M\op$ for all $n$. In particular, $\xi p_n \in L^2(M \ovt M\op) = L^2(M) \ot L^2(M\op)$ and we can take separable Hilbert subspaces $H_n \subset L^2(M)$, $K_n \subset L^2(M\op)$ such that $\xi p_n \in H_n \ot K_n$. We then find countable subsets $V_n \subset M$ such that for every $n$, the vector $\xi p_n$ belongs to the $\|\,\cdot\,\|_2$-closed linear span of $\{a \ot b\op \mid a,b \in V_n\}$. Defining $N$ as the von Neumann subalgebra of $M$ generated by all the sets $V_n$, our statement is proven.

{\it Let $N_1 \subset M$ be a von Neumann subalgebra with separable predual. Then there exists a von Neumann subalgebra $N_2 \subset M$ with separable predual such that $N_1 \subset N_2$ and $\delta(N_1) \subset \aff(N_2 \ovt N_2\op)$.} Take a separable and weakly dense $C^*$-subalgebra $B_1 \subset N_1$. By the previous paragraph and because $\delta$ is norm-measure continuous, we can take $N_2 \subset M$ with separable predual such that $\delta(B_1) \subset \aff(N_2 \ovt N_2\op)$. Replacing $N_2$ by the von Neumann algebra generated by $N_1$ and $N_2$, we may assume that $N_1 \subset N_2$. Since $\delta(B_1) \subset \aff(N_2 \ovt N_2\op)$, it follows from \cite[Lemma 4.2 and Theorem 4.3]{Th06} that $\delta(N_1) \subset \aff(N_2 \ovt N_2\op)$.

{\it Let $N_1 \subset M$ be a von Neumann subalgebra with separable predual. Then there exists a von Neumann subalgebra $N \subset M$ with separable predual such that $N_1 \subset N$ and $\delta(N) \subset \aff(N \ovt N\op)$.} Using the previous paragraph, we inductively find an increasing sequence of von Neumann subalgebras $N_1 \subset N_2 \subset \cdots$ with separable predual such that $\delta(N_k) \subset \aff(N_{k+1} \ovt N_{k+1}\op)$ for all $k$. We define $N$ as the von Neumann algebra generated by all the $N_k$. By construction, $N$ has separable predual and $\delta(N_k) \subset \aff(N \ovt N\op)$ for all $k$. Again using \cite[Lemma 4.2 and Theorem 4.3]{Th06}, it follows that $\delta(N) \subset \aff(N \ovt N\op)$.

We can now conclude the proof of step~2. Since $M$ is diffuse, we can fix a diffuse abelian von Neumann subalgebra $A \subset M$ with separable predual. By \cite[Theorem 6.4]{Th06}, we can replace $\delta$ by $\delta - \partial \xi$ for some $\xi \in \aff(M \ovt M\op)$ and assume that $\delta(a) = 0$ for all $a \in A$. We prove that $\delta(x) = 0$ for all $x \in M$. Fix an arbitrary $x \in M$. Define $N_1$ as the von Neumann algebra generated by $A$ and $x$. Note that $N_1$ has a separable predual. By the previous paragraph, we can take a von Neumann subalgebra $N \subset M$ with separable predual such that $N_1 \subset N$ and $\delta(N) \subset \aff(N \ovt N\op)$. By the initial assumption of step~2, the restriction of $\delta$ to $N$ is inner. So we can take a $\xi \in \aff(N \ovt N\op)$ such that $\delta(y) = (\partial \xi)(y)$ for all $y \in N$. Since $A \subset N$ and $\delta(a)=0$ for all $a \in A$, it follows that $(a \ot 1)\xi = (1 \ot a\op)\xi$ for all $a \in A$. Since $A$ is diffuse, this implies that $\xi = 0$. Since $x \in N$, it then follows that $\delta(x) = 0$. This concludes the proof of step~2.

{\it Step 3. Proof of the lemma~: it suffices to prove Theorem \ref{thm.inner} when $M$ is a II$_1$ factor with separable predual.} Using step~2, we may already assume that $M$ is diffuse and has separable predual. Let $p_0 \in \cZ(M)$ be the maximal projection such that $\cZ(M) p_0$ is diffuse (possibly, $p_0=0$). Let $p_1,p_2,\ldots$ be the minimal projections in $\cZ(M)(1-p_0)$. Note that $\sum_{n=0}^\infty p_n = 1$. As in the proof of step~1, we may assume that $\delta(M p_n) \subset \aff(Mp_n \ovt (M p_n)\op)$ for all $n$. Denote by $\delta_n$ the restriction of $\delta$ to $M p_n$. Since $M p_0$ has a diffuse center, it follows from \cite[Theorem 6.4]{Th06} that $\delta_0$ is inner. For all $n \geq 1$, we have that $M p_n$ is a II$_1$ factor with separable predual. So by assumption, all $\delta_n$, $n \geq 1$, are inner. But then also $\delta$ is inner.
\end{proof}

\begin{proof}[Proof of Theorem \ref{thm.inner}]
Using Lemma \ref{lem.reduction}, it suffices to take a II$_1$ factor $M$ with separable predual and a derivation $\delta : M \recht \aff(M \ovt M\op)$ that is continuous from the norm topology on $M$ to the measure topology on $\aff(M \ovt M\op)$. Denote by $\tau$ the unique tracial state on $M$. By \cite[Corollary 4.1]{Po81}, we can fix a copy of the hyperfinite II$_1$ factor $R \subset M$ such that $R'\cap M = \C 1$. By \cite[Theorem 6.4]{Th06}, we can replace $\delta$ by $\delta - \partial \xi$ and assume that $\delta(x) = 0$ for all $x \in R$. We prove that $\delta = 0$. Fix a unitary $u \in \cU(M)$ with $\tau(u) = 0$. It suffices to prove that $\delta(u) = 0$.

Fix a free ultrafilter $\om$ on $\N$ and consider the ultrapower $M^\om$. By \cite[Corollary on p.\ 187]{Po92}, choose a unitary $v \in R^\om$ such that the subalgebras $v^k M v^{-k} \subset M^\om$, $k \in \Z$, are free. Fix a Haar unitary $a \in \cU(R)$, i.e.\ a unitary satisfying $\tau(a^m) = 0$ for all $m \in \Z \setminus \{0\}$. Define, for $k \geq 1$, $a_k = v^k a v^{-k}$. It follows that $a_k \in \cU(R^\om)$ are $*$-free Haar unitaries that are moreover free w.r.t.\ $M$. Write $a_k = (a_{k,n})$ with $a_{k,n} \in \cU(R)$.

Similar to the definition of $x_n$ in the proof of  \cite[Proposition 3.1]{Al13},
we consider for all large $n$ and $m$, the unitary
$$w_{m,n} = a_{1,n} u \, a_{2,n} u \, \cdots \, a_{m,n} u$$
and prove that either $\delta(u)=0$ or $\delta(w_{m,n})$ is ``very large almost everywhere'', contradicting the continuity of $\delta$.

Since $\delta(x) = 0$ for all $x \in R$, we get that
\begin{equation}\label{sterreke}
\delta(w_{m,n}) = \Bigl(\sum_{k=1}^m a_{1,n} u \, a_{2,n} \, \cdots a_{k-1,n} u \, a_{k,n} \ot \bigl(a_{k+1,n} u \cdots a_{m,n} u\bigr)\op\Bigr) \, \delta(u) \; ,
\end{equation}
where, by convention, the first term in the sum is $a_{1,n} \ot (a_{2,n} u \cdots a_{m,n} u)\op$ and the last term is $a_{1,n} u \cdots a_{m-1,n} u \, a_{m,n} \ot 1$.

Consider, in the ultrapower $(M \ovt M\op)^\om$, the element
$$T_m = \sum_{k=1}^m a_{1} u \, a_{2} u \, \cdots a_{k-1} u \, a_k \ot \bigl(a_{k+1} u \cdots a_m u\bigr)\op \; .$$
We claim that $T_m$ is the sum of $m$ $*$-free Haar unitaries. To prove this, it suffices to show that the first tensor factors $a_1 , a_1 u a_2, a_1 u a_2 u a_3, \ldots$ form a $*$-free family of Haar unitaries. Since $a_1,a_2,a_3,\ldots$ is a $*$-free family of Haar unitaries that are $*$-free w.r.t.\ $u$, also the Haar unitaries $a_1, ua_2, ua_3, ua_4,\ldots$ are $*$-free. But then the conclusion follows by taking the product of the first $k$ unitaries in this last sequence, again producing a $*$-free family of Haar unitaries.

Since $T_m$ is the sum of $m$ $*$-free Haar unitaries, we get from \cite[Example 5.5]{HL99} an explicit formula for the spectral distribution of $|T_m|$. It follows that $|T_m|$ has the same distribution as $2 \sqrt{m-1} \, S_m$, where $S_m$ is a sequence of random variables satisfying $0 \leq S_m \leq 1$ and converging in distribution to the normalized quarter circle law. Therefore, the spectral projections
$$q_m = \mathbf{1}_{[\sqrt{m},+\infty)}(T_m^* T_m) = \mathbf{1}_{[m^{1/4},+\infty)}(\, |T_m| \,) = \mathbf{1}_{\bigl[ m^{1/4} (4(m-1))^{-1/2} ,+\infty \bigr)}(S_m)$$
satisfy $\lim_m \tau(q_m) = 1$. Write $q_m = (q_{m,n})$ where $q_{m,n}$ are projections in $M \ovt M\op$.

Fix an arbitrary $\eps > 0$. Since $\delta$ is continuous, fix $\rho > 0$ such that $\delta(z) \in B(\tau,\eps/2)$ whenever $z \in M$ and $\|z\| < \rho$. Take $m$ large enough such that $m^{-1/4} < \rho$ and $\tau(q_m) > 1-\eps$. For every $n$, the element $m^{-1/4} w_{m,n}$ has norm less than $\rho$. Therefore, $\delta(m^{-1/4} w_{m,n})$ belongs to $B(\tau,\eps/2)$ and we find a projection $p_n \in M \ovt M\op$ with
$$\tau(p_n) > 1-\eps/2 \quad\text{and}\quad \bigl\|\delta\bigl(m^{-1/4} w_{m,n}\bigr) \, p_n \bigr\| < \eps/2 \; .$$
We also fix a projection $e_0 \in M \ovt M\op$ with $\tau(e_0) > 1 - \eps/2$ and such that $\delta(u) e_0 \in M \ovt M\op$. We write $e_n = e_0 \wedge p_n$ and view $e = (e_n)$ as a projection in $(M \ovt M\op)^\omega$. By \eqref{sterreke}, we have in $(M \ovt M\op)^\omega$ the equality $(\delta(w_{m,n}) e_0)_n = T_m \, (\delta(u) e_0)$, and therefore also that $(\delta(w_{m,n}) e_n)_n = T_m \, (\delta(u) e_0) \, e$. We then find that
\begin{align*}
\eps^2 &> \lim_{n \recht \omega} \bigl\|\delta\bigl(m^{-1/4} w_{m,n}\bigr) \, e_n \bigr\|^2 \\
&\geq \lim_{n \recht \omega} \bigl\|\delta\bigl(m^{-1/4} w_{m,n}\bigr) \, e_n \bigr\|_2^2 \\
&= \tau\bigl( e \; (\delta(u)e_0)^* \; m^{-1/2} T_m^* T_m \; (\delta(u)e_0) \; e \bigr) \\
&\geq \tau\bigl( e \; (\delta(u)e_0)^* \; q_m \; (\delta(u)e_0) \; e \bigr) \; .
\end{align*}
Since $\tau(q_m) > 1-\eps$, we can fix $n$ such that
$$\|q_{m,n} \, \delta(u) \, e_n \|_2 < \eps \quad\text{and}\quad \tau(q_{m,n}) > 1 - \eps \; .$$
Since $\tau(e_n) > 1-\eps$, we have proven that for every $\eps > 0$, there exist projections $p,q \in M \ovt M\op$ such that $\tau(p) > 1-\eps$, $\tau(q) > 1-\eps$ and $\|q \delta(u) p\|_2 < \eps$. This means that $\delta(u) = 0$.
\end{proof}

The proof of Theorem \ref{thm.inner} gives no indication as to whether or not the Connes-Shlyakhtenko first $L^2$-cohomology vanishes as well. Note however that in order for a first $L^2$-cohomology theory to ``work well'' for II$_1$ factors $M$, the corresponding derivations should be uniquely determined by their values on a set of elements generating $M$ as a von Neumann algebra. In order for this to be the case, the derivations should normally satisfy some continuity property, even if that continuity is ``very weak''. However, by combining Theorem \ref{thm.inner} with the closed graph theorem, it follows that any derivation from $M$ into $\aff(M \ovt M\op)$ that satisfies some ``reasonable''  weak continuity property, must in fact be inner (see also
Remark \ref{rem} hereafter):

\begin{corollary}\label{cor}
Let $M$ be a finite von Neumann algebra and write $\cE = \aff(M \ovt M\op)$. Assume that $\delta : M \recht \cE$ is a derivation. If $\delta$ has a closed graph for the norm topology on $M$ and the measure topology on $\cE$, then $\delta$ is inner.

This is in particular the case if $\delta$ is norm-$\cT$-continuous w.r.t.\ any vector space topology $\cT$ on $\cE$ satisfying the following two properties:
\begin{itemize}
\item the inclusion $M \ovt M\op \recht \cE$ is norm-$\cT$-continuous;
\item for every fixed $a \in M \ovt M\op$, the map $\cE \recht \cE : \xi \mapsto \xi a$ is $\cT$-$\cT$-continuous.
\end{itemize}
\end{corollary}

\begin{proof}
Take an orthogonal family of projections $p_i \in \cZ(M)$ such that $\sum_{i \in I} p_i = 1$ and every $M p_i$ is countably decomposable. As in step~1 of Lemma \ref{lem.reduction}, we may assume that $\delta(M p_i) \subset \aff(M p_i \ovt (M p_i)\op)$. If $\delta$ has closed graph, the restrictions $\delta_i$ of $\delta$ to $M p_i$ still have closed graph. If all these restrictions $\delta_i$ are inner, also $\delta$ is inner. So to prove the first part of the corollary, we may assume that $M$ admits a faithful normal tracial state $\tau$ that we keep fixed. But then, the formula
$$d(\xi,\eta) = \inf \{ \eps > 0 \mid \exists \; \text{projection}\; p \in M \ovt M\op : \tau(1-p) < \eps \; \text{and} \; \|(\xi-\eta)p\| < \eps \}$$
defines a translation invariant, complete metric on $\aff(M \ovt M\op)$ that induces the measure topology. So by \cite[2.15]{Ru91}, the closed graph theorem is valid and we find that $\delta$ is continuous, and hence inner by Theorem \ref{thm.inner}.

Assume now in general that $\delta$ is norm-$\cT$-continuous. To prove that $\delta$ has closed graph, assume that $\|x_n\| \recht 0$ and $\delta(x_n) \recht \xi$ in the measure topology. We have to show that $\xi = 0$. Fix a normal tracial state $\tau$ on $M$. Choose projections $p_n \in M \ovt M\op$ such that $\tau(p_n) > 1-2^{-n}$ and $\|(\delta(x_n) - \xi)p_n\| < 1/n$. Define the projections $q_k = \bigwedge_{n \geq k} p_n$ and note that $q_k$ is increasing and satisfies $\tau(q_k) \recht 1$. For every fixed $k$, we get that $(\delta(x_n) - \xi)q_k$ converges to $0$ in norm, as $n \recht \infty$. By the first assumption on $\cT$, this convergence also holds in $\cT$. But also $\delta(x_n) \recht 0$ in $\cT$ so that, by our second assumption on $\cT$, the sequence $\delta(x_n) q_k$ converges to $0$ in $\cT$ as $n \recht \infty$. We conclude that $\xi q_k = 0$ for all $k$. This implies that $\xi z_\tau = 0$ where $z_\tau \in \cZ(M)$ is the support projection of $\tau$. Since $\tau$ was arbitrary, it follows that $\xi = 0$.
\end{proof}

\begin{remark}\label{rem}
We should point out that we have no concrete examples of vector topologies on $\aff(M \ovt M\op)$ satisfying the conditions in the second part of Corollary \ref{cor} and that are strictly weaker than the measure topology (in fact, it is not even clear whether such a topology exists!).
Let us also point out that there are other weak continuity properties of $\delta$ implying that $\delta$ has closed graph, thus following inner by the first part of Corollary \ref{cor}. For instance, by using a similar argument as above, one can easily prove that this is the case when $\delta$ satisfies the following weak continuity property: whenever $x_n$ is a sequence in $M$ such that $\|x_n\| \recht 0$, there exists a sequence of projections $p_n \in M \ovt M\op$ such that $p_n \recht 1$ strongly, $\delta(x_n) p_n \in M \ovt M\op$ and $\delta(x_n) p_n \recht 0$ $\sigma$-weakly.
\end{remark}

\end{document}